\documentclass[12pt]{amsart}

\usepackage{amscd,amssymb,url,colonequals,enumitem,xy}
\xyoption{all}

\newtheorem{theorem}{Theorem}% [section]
\newtheorem{lemma}[theorem]{Lemma}
\newtheorem{proposition}[theorem]{Proposition}

\theoremstyle{definition}

\theoremstyle{remark}
\newtheorem{remark}[theorem]{Remark}

\newtheorem{hj-local}[theorem]{Local Hermite-Joubert Problem}

\newcommand{\Alt}{\operatorname{A}}
\newcommand{\GL}{\operatorname{GL}}

\newcommand{\PGL}{\operatorname{PGL}}
\newcommand{\trdeg}{{\rm trdeg}}
\newcommand{\ed}{{\rm ed}}

\newcommand{\bbZ}{{\mathbb Z}}

\newcommand{\Z}{{\mathbb Z}}

\newcommand{\Sym}{\operatorname{S}}

\renewcommand{\P}{{\mathbb P}}
\newcommand{\bbP}{{\mathbb P}}

\newcommand{\G}{\operatorname{G}}

\newcommand{\Spec}{{\rm Spec}}
\newcommand{\cchar}{{\rm char}}

\begin{document}

\keywords{Hermite's theorem, quintic polynomial, Galois cohomology, Tsen-Lang theorem, essential dimension}
\subjclass[2010]{12G05, 14G05}
%  
% 12  View Publications (1959-now) Field theory and polynomials
% 12F10 (1973-now) Separable extensions, Galois theory 
% 12G  View Publications (1973-now) Homological methods (field theory)
% 12G05  View Publications (1973-now) Galois cohomology [See also 14F22, 16Hxx, 16K50]
% 	14G05   	Rational points
% 14J70   	Hypersurfaces
% 11G05 View Publications (1980-now) Elliptic curves over global field

\title[Hermite's theorem]{Hermite's theorem via Galois cohomology}

\author[Matthew Brassil and Zinovy Reichstein]{Matthew Brassil 
and Zinovy Reichstein}

\address
{Department of Mathematics \\
University of British Columbia \\
Vancouver
\\
CANADA}

\email{mbrassil@math.ubc.ca, reichst@math.ubc.ca}

\thanks
{Zinovy Reichstein was partially supported by
National Sciences and Engineering Research Council of
Canada Discovery grant 253424-2017.}

\begin{abstract} 
An 1861 theorem of Hermite asserts that for every field extension 
$E/F$ of degree $5$ there exists an element of $E$ whose 
minimal polynomial over $F$ is of the form $f(x) = x^5 + c_2 x^3 + c_4 x + c_5$ for some $c_2, c_4, c_5 \in F$.
We give a new proof of this theorem using techniques of Galois cohomology, 
under a mild assumption on $F$.
\end{abstract} 

\maketitle
% \tableofcontents

\section{Introduction}
An 1861 theorem of Hermite asserts that 
for every field extension $E/F$ of degree $5$
there exists an element of $E$ whose minimal polynomial over $F$ is of the form
\[ \text{$f(x) = x^5 + c_2 x^3 + c_4 x + c_5$ for some $c_2, c_4, c_5 \in F$.} \]
Modern proofs of this result have been given by Coray~\cite{coray2} and
Kraft~\cite{kraft}. Coray's proof relies on techniques of arithmetic geometry, 
whereas Kraft's is based on representation theory, in the spirit of Hermite's original paper~\cite{hermite}.
The purpose of this note is to give yet another proof of the following variant of Hermite's theorem 
using techniques of Galois cohomology.
% under a mild assumption on $F$. 
% The remainder of this note will be devoted to proving the following variant of Hermite's theorem.

\begin{theorem}\label{thm.main}
Let $E/F$ be a field extension of degree $5$. Assume $F$ contains an algebraically closed field $k$.
Then there exists an element $a \in E$ whose minimal polynomial is of the form
$f(x) = x^5 + c_2 x^3 + c_4 x + c_5$ for some $c_2, c_4, c_5 \in F$.
\end{theorem}

\section{A geometric restatement of the problem}

Clearly every $z \in E \setminus F$ is a primitive element for $E/F$. Choose one such $z$
and set $a = x_0 + x_1 z + \dots + x_4 z^4 \in E$,
where $x_0, \dots, x_4 \in F$ are to be specified later.
The characteristic polynomial of $a$ over $F$ is $f(t) = \det(t \cdot 1_F - a) = t^5 + c_1 t^4 + \dots c_4 t + c_5$,
where $t$ is a commuting variable, and $\det$ denotes the norm in the field extension
$E(t)/F(t)$. Each $c_i$ is a homogeneous polynomial of degree $i$
in $x_0, \dots, x_4$ with coefficients in $F$. We are interested in non-trivial solutions of the system
\begin{equation} \label{e.c}
c_1(x_0, \dots, x_4) = c_3(x_0, \dots, x_4) = 0
\end{equation}
in $\bbP^4(F)$. Note that $c_1$ cuts out a linear 
subvariety $\P^3 \subset \bbP^4$, and $c_3$ cuts out a cubic surface in this $\P^3$ defined over $F$.
We will denote this cubic surface by $X$. 
Any solution $(x_0: \ldots : x_4) \in \P^4(F)$ to~\eqref{e.c} (or equivalently, any $F$-point of $X$)
gives rise to an element $a = x_0 + x_1 z + \ldots + x_4 z^4 \in E$ whose characteristic polynomial is 
of the desired form. Moreover, if $(x_0: \ldots: x_4) \neq (1: 0: \ldots: 0)$, then $a$ is a primitive 
element of $E$, so its minimal polynomial is the same as its characteristic polynomial.

\begin{lemma} \label{lem1} (a) $(1:0:0:0:0)$ is a solution to~\eqref{e.c} if and only if $\cchar(F) = 5$.

\smallskip
(b) In the course of proving Theorem~\ref{thm.main}, we may assume without loss of generality that $\cchar(F) \neq 5$.
In particular, we may assume that $E/F$ is separable.

\smallskip
(c) In order to prove Theorem~\ref{thm.main}, it suffices to show that
the system~\eqref{e.c} has a non-trivial solution in $F$ (or equivalently, $X$ has an $F$-point).
\end{lemma}

\begin{proof} (a) If $x_0 = 1$ and $x_1 = \dots = x_4 = 0$, then $a = 1$, and
\[ \det(t \cdot 1_F - a) = (t-1)^5 = t^5 - 5 t^4 + 10 t^3 - 10 t^2 + 5 t - 1 , \]
so $c_1 = -5$ and $c_3 = -10$. Thus $c_1 = c_3 = 0$ if and only if $\cchar(F) = 5$.

(b) An easy application of the Jacobian criterion, shows that
the surface $X$ is smooth whenever $\cchar(k) \neq 3$; see \cite[Lemma 1.2]{coray2}.

Assume $\cchar(F) = 5$. By part (a), $(1: 0: \dots: 0)$ is an $F$-point of $X$.
Consequently, by~\cite[Theorem 1.1]{kollar}, $X$ is unirational over $F$. 
Since $F$ is an infinite field (recall that we are assuming that $F$ contains an algebraically closed field),
this tells us that $F$-points are dense in $X$. In particular, there is an $F$-point on $X$, 
other than $(1:0:0:0:0)$, as desired.

(c) By part (b), we may assume that $\cchar(k) \neq 5$. By part (a), $(1:0 :0:0:0)$ is not 
a solution to~\eqref{e.c}. Thus any solution gives rise to $a = x_0 + x_1 z + \ldots + x_4 z^4 \in E$ 
whose minimal polynomial has the desired form.
\end{proof}

\section{Preliminaries on Galois cohomology and essential dimension}

In this section we give a brief summary of the background material on Galois cohomology and essential dimension, which will be used in the sequel.
For details we refer the reader to~\cite{serre-gc}, \cite[Chapter I]{serre-ci}, \cite{berhuy-favi1}, and \cite{icm}. 

\smallskip
\begin{itemize} 
\item
Let $G$ be a smooth algebraic group over $k$ and $F$ be a field. The Galois cohomology set $H^1(F, G)$ is in a natural bijective correspondence
with isomorphism classes of $G$-torsors $T \to \Spec(F)$. The class of the split torsor $G \times_{\Spec(k)} \Spec(F) \to \Spec(F)$ is usually denoted by $1 \in H^1(F, G)$. 

\smallskip
\item
In the case where $G$ is the symmetric group $\Sym_n$ (viewed as a constant finite group over $k$), the Galois cohomology set $H^1(F, \Sym_n)$ 
is also in a natural bijective correspondence with isomorphism classes of $n$-dimensional \'etale algebras $E/F$. Recall that an \'etale algebra $E$ is,
by definition, is a direct product of the form $E = E_1 \times \ldots \times E_r$, where each $E_i/F$ is a finite separable field extension.

\smallskip
\item
In particular, a separable field extension $E/F$ of degree $n$ gives rise to a class in $H^1(F, \Sym_n)$.
This class lies in the image of the natural map $H^1(F, G) \to H^1(F, \Sym_n)$, for a subgroup $G$ of $\Sym_n$ if and
only if the Galois group of $E/F$ is contained in $G$.

\smallskip
\item
Let $E/F$ be a finite field extension, and $k \subset F$ be a subfield. We say that $E/F$ descends to an intermediate extension $k \subset F_0 \subset F$ if
$E = E_0 \otimes_{F_0} F$ for some field extension $E_0/F_0$. The essential dimension $\ed(E/F)$ is the minimal transcendence degree $\trdeg_k(F_0)$ such that
$E/F$ descends to $F_0$. This number depends on the base field $k$, which we assume to be fixed throughout.

\smallskip
\item
If $G$ is an algebraic group over $k$, then the essential dimension $\ed(\tau)$ of a $G$-torsor $\tau \colon T \to \Spec(F)$ is defined in a similar manner.
We say that $\tau$ descends to a subfield $F_0 \subset F$ if it lies in the image of the natural map $H^1(F_0, G) \to H^1(F, G)$. The essential dimension $\ed(\tau)$
is the minimal transcendence degree of an intermediate extension $k \subset F_0 \subset F$ such that $\tau$ descends to $F_0$.  In particular, if $E/F$ is a separable field extension
of degree $n$, and $\tau \in H^1(F, \Sym_n)$ is the class corresponding to $E/F$, then $\ed(\tau) = \ed(E/F)$.

\smallskip
\item
The essential dimension $\ed(G)$ of the group $G$ is the maximal value of $\ed(\tau)$ over all extensions $F/k$ and all $G$-torsors 
$\tau$ over $\Spec(F)$. In the case where $G$ is the symmetric group $\Sym_n$, $\ed(\Sym_n)$ is the maximal value of $\ed(E/F)$, 
where $E/F$ ranges over all separable field extensions of degree $n$.
\end{itemize}

\section{First steps towards the proof of the main theorem}

Recall that a field $F$ has property $C_i$ (or $F$ is a $C_i$-field, for short) if every system 
\[ p_1(x_1, \dots, x_n) = \ldots = p_r(x_1, \dots, x_n) = 0 \]
of homogeneous polynomials
has a non-trivial solution, provided that ${\rm deg}(p_1)^i + \dots + {\rm deg}(p_r)^i < n$;
%if $p_1, \dots, p_r$ are homogeneous polynomials of degree $d_1, \dots, d_r$,
%respectively, and $d_1^i + \dots + d_r^i < n$; 
see~\cite[Definition 5.1.1]{pfister}.

\begin{lemma} \label{lem2} (a) Theorem~\ref{thm.main} holds if $\trdeg_k(F) \leqslant 1$.

\smallskip
(b) More generally, Theorem~\ref{thm.main} holds if $\ed(E/F) \leqslant 1$.
\end{lemma}

\begin{proof} (a) By the Tsen-Lang Theorem~\cite[Corollary 5.1.5]{pfister}, $F$ is a $C_1$-field. In particular, every cubic surface in $\P^3$ defined over $F$ has an $F$-point.
Now apply Lemma~\ref{lem1}(c).

(b) By the definition of $\ed(E/F)$, there exists an intermediate field $k \subset F_0 \subset F$ and a field extension $E_0/F_0$ of degree $5$ such that
$\trdeg_k(F_0) \leqslant 1$ and $E = E_0 \otimes_{F_0} F$. Note that for any $a \in E_0$, the minimal polynomials of $a$ over $F$ and $F_0$ coincide.
Thus we may replace $E/F$ by $E_0/F_0$ and look for an element $a$ with desired properties in $E_0$. (b) now follows from (a).
\end{proof}

Recall that a field extension $F'/F$ is called constructible if there exists a tower of fields 
\[ F = F_0 \subset F_1 \subset \dots \subset F_n = F' ,\]
such that $[F_{i}: F_{i-1}] = 2$ for every $i = 1, \dots, n$.\footnote{The term ``constructible" is related to the classical theory of ruler and compass constructions,
where towers of quadratic field extensions appear naturally; see, e.g.,~\cite[Section 13.4]{artin}.}

\begin{lemma} \label{lem3} Let $F'/F$ be a field extension and $E' = E \otimes_F F'$. 

(a) Assume $[F':F] = 2$. If Theorem~\ref{thm.main} holds for $E'/F'$, then it holds for $E/F$.

(b) Assume $F'/F$ is constructible.  If Theorem~\ref{thm.main} holds for $E'/F'$, then it holds for $E/F$.
\end{lemma}

\begin{proof} (a) By~\cite[Proposition 2.2]{coray2}, our cubic surface $X$ has an $F'$-point if and only if $X$ has an $F$-point.
The desired conclusion now follows from Lemma~\ref{lem1}(c). To prove (b), apply (a) recursively.
\end{proof}

Note that $\ed(\Sym_5) = 2$ if $\cchar(k) \neq 2$; see~\cite[Theorem 6.5]{br}. This means that for some separable extensions $E/F$ of degree $5$, $\ed(E/F) = 2$
and consequently Lemma~\ref{lem2}(b) does not apply. However, in view of Lemma~\ref{lem3} it suffices to prove the following.

\begin{proposition} \label{prop1} For every separable field extension $E/F$ of degree $5$, there exists a constructible extension $F'/F$ such that
$\ed(E'/F') \leqslant 1$. Here $E' = E \otimes_F F'$.
\end{proposition}

\begin{remark} Since $k$ is algebraically closed, it is easy to see that $\ed(E/F) \geqslant 1$ for every non-trivial 
finite field extension $E/F$ (with $E \neq F$).
% we note that since $k$ is algebraically closed, the only extension $F/k$ of transcendence degree $0$ is $F = k$, in which case
% a field extension $E/F$ with $[E:F] = 5$ cannot exist. 
Thus $\leqslant 1$ can be replaced by $= 1$ in both parts of Lemma~\ref{lem2} and in Proposition~\ref{prop1}. 
\end{remark}

\section{Conclusion of the proof of the main theorem}

In this section we will complete the proof of Theorem~\ref{thm.main} by establishing Proposition~\ref{prop1}.

% We claim that after passing to a quadratic extension $F'/F$, we may assume that the Galois group of $E/F$ is contained in $\Alt_5$.
Let $\alpha$ denote the class of the field extension $E/F$ in $H^1(F, \Sym_5)$.
% is contained in the image of the natural map $H^1(F, \Alt_5) \to H^1(F, \Sym_5)$.
Consider the exact sequence
\[  \xymatrix{  1 \ar@{->}[r] & \Alt_5  \ar@{->}[r] & \Sym_5 \ar@{->}[r]^{{\rm sign} \; \; \;} & \Z /2 \Z  \ar@{->}[r] & 1}, \]
and the associated sequence 
\[  \xymatrix{  H^1(F, \Alt_5) \ar@{->}[r] & H^1(F, \Sym_5)  \ar@{->}[r]^{D \; \; } & H^1(F, \Z/ 2\Z)} \]
of Galois cohomology sets; cf.~\cite[Section 5.5]{serre-gc}. Here, as usual,  $\Alt_5$ denotes the alternating subgroup of $\Sym_5$. 
(If $\cchar(F) \neq 2$, $D(\alpha)$ is just the discriminant of $E/F$, viewed as
an element of $H^1(F, \Z/ 2 \Z) \simeq F^*/(F^*)^2$.)
The class $D(\alpha) \in H^1(F, \bbZ/2 \bbZ)$ is represented by a separable quadratic extension $F'/F$.
After replacing $F$ by $F'$, we may assume that $\alpha \in H^1(F, \Sym_5)$ lies in the image of $H^1(F, \Alt_5)$. Equivalently,
we may assume without loss of generality that the Galois group of $E/F$ is a subgroup $\Alt_5$.
 
Note that this reduction does not, by itself, allow us to conclude that $\ed(E/F) \leqslant 1$. Indeed, $\ed(\Alt_5) = 2$, 
assuming $\cchar(k) \neq 2$; see~\cite[Theorem 6.7]{br}. We will need to pass to a further constructible extension $F'/F$ 
in order to ensure that $\ed(E/F) \leqslant 1$.

For notational simplicity, we will continue to denote the class of $E/F$ in $H^1(F, \Alt_5)$ by $\alpha$. 
Since $k$ is algebraically closed, $\Alt_5$ can be embedded in $\PGL_2(k)$; see~\cite[p.19-04]{serre-A5}.\footnote{In the case where $\cchar(k) \neq 2, 3$ or $5$, see also~\cite[Proposition 1.1(3)]{beauville}.}
Let us now consider the commutative diagram
\[  \xymatrix{  1 \ar@{->}[r] & \G_m \ar@{->}[r] & \GL_2 \ar@{->}[r] & \PGL_2  \ar@{->}[r] & 1 \\ 
1 \ar@{->}[r] & \G_m \ar@{->}[r] \ar@{=}[u] & G \ar@{->}[r] \ar@{^{(}->}[u] & \Alt_5  \ar@{->}[r] \ar@{^{(}->}[u] & 1} \] 
of algebraic groups over $k$, where $G$ is the preimage of $\Alt_5$ in $\GL_2$. This diagram induces a commutative diagram 
\[  \xymatrix{ % H^1(F, \GL_2) \ar@{->}[r] 
 & H^1(F, \PGL_2)  \ar@{->}[r]^{\delta} & H^2(F, \G_m) \\ 
H^1(F, G) \ar@{->}[r]^{\pi_F} % \ar@{-}[u]
& H^1(F, \Alt_5) \ar@{->}[r]^{\delta}  \ar@{->}[u] & H^2(F, \G_m) \ar@{=}[u]} \] 
of Galois cohomology sets, where the bottom row is exact; cf.~\cite[Section 5.5]{serre-gc}. Here $\delta$ denotes the connecting map.
The class of $\delta(\alpha)$ is represented by a quaternion algebra over $F$. This algebra can be split by a quadratic extension $F'/F$. After replacing
$F$ by $F'$, we may assume that $\delta(\alpha) = 0$. Equivalently, $\alpha = \pi_F(\beta)$ for some $\beta \in H^1(F, G)$. 

Since $\dim(G) = 1$ and the natural $2$-dimensional representation of $G$ is generically free (i.e., the stabilizer of a general point in trivial), one readily concludes that
$\ed(G) = 1$; see~\cite[Proposition 2.4]{berhuy-favi2}. Consequently, $\ed(\beta) \leqslant 1$ and thus $\ed(\alpha) \leqslant 1$.
This completes the proof of Proposition~\ref{prop1} and thus of Theorem~\ref{thm.main}. 
\qed

\begin{remark} The condition on $F$ in Theorem~\ref{thm.main} can be weakened slightly: our argument goes through, with only minor changes, under the assumption that
$F$ is a $p$-field for some prime $p \neq 3$ (not necessarily algebraically closed). Recall that a field $k$ is called a $p$-field 
if $[l: k]$ is a power of $p$ for every finite field extension $l/k$; see \cite[Definition 4.1.11]{pfister}.\footnote{Some authors use the terms ``$p$-closed field" or "$p$-special field" in place of "$p$-field".}
\end{remark}

\begin{remark} Proposition~\ref{prop1} fails for separable field extensions of degree $n \geqslant 6$. In fact, if $E/F$ is a general extension of degree $n$, then
for any constructible extension $F'/F$, 
\[ \ed(E'/F') \geqslant \ed(\Sym_n; 3) = \lfloor \frac{n}{3} \rfloor . \]
Here $\ed(\Sym_n; 3)$ denotes the essential dimension of $\Sym_n$ at $3$, $\lfloor \dfrac{n}{3} \rfloor$ denotes the integer part of $\dfrac{n}{3}$, and we are assuming that $\cchar(k) \neq 3$; see~\cite[Corollary 4.2]{mr}.
\end{remark} 

\begin{remark} Generalizing Hermite's theorem to field extensions of degree $n \geqslant 6$ is an interesting and largely open problem. 
The only known positive result in this direction is the classical theorem of Joubert~\cite{joubert} for $n = 6$. There are also negative 
results for some $n$. For an overview, see~\cite{brr}.
\end{remark}

\end{document}